%% file: paper.tex
\documentclass[10pt,twoside,reqno]{amsart}

\usepackage[dvipsnames]{xcolor}
\usepackage{amssymb,amsmath,amstext,amsthm,amsfonts,amscd, mathtools}
\usepackage{enumerate}

\usepackage[ansinew]{inputenc} 
\usepackage{graphicx}
\usepackage[mathscr]{eucal}
\usepackage{hyperref}

\AtBeginDocument{%
   \def\MR#1{}
}

\newcommand{\blob}{\rule[.2ex]{.8ex}{.8ex}}

\newcommand{\bigo}{\mathcal{O}}

\newcommand\smallo{
  \mathchoice
    {{\scriptstyle\mathcal{O}}}
    {{\scriptstyle\mathcal{O}}}
    {{\scriptscriptstyle\mathcal{O}}}
    {\scalebox{.25}{$\scriptscriptstyle\mathcal{O}$}}
  }

\newcommand{\R}{\mathbb{R}}
\newcommand{\C}{\mathbb{C}}
\newcommand{\N}{\mathbb{N}}
\newcommand{\Z}{\mathbb{Z}}
\newcommand{\Q}{\mathbb{Q}}
\newcommand{\T}{\mathbb{T}}

\newcommand{\Jd}{\textbf{J}}

\newcommand{\DC}{{\rm DC} (\T^d)}
\newcommand{\sDC}{{\rm DC} (\T)}

\newcommand{\moc}{w}
\newcommand{\bsum}[1]{\obs^{({#1})} \,} 
\newcommand{\bsumap}[1]{\obs_n^{({#1})} \,} 
\newcommand{\bsumexpk}[1]{e_{\kb}^{({#1})} \,} 
\newcommand{\obs}{\phi}
\newcommand{\fc}[1]{\widehat{#1}} 
\newcommand{\esum}{\mathscr{E}}

\newcommand{\x}{\boldsymbol{x}}

\newcommand{\kb}{\boldsymbol{k}}

\newcommand{\omb}{\boldsymbol{\om}}

\newcommand{\abs}[1]{\left| #1 \right|} 
\newcommand{\norm}[1]{\left\|#1\right\|} 
\newcommand{\normT}[1]{\left\|#1\right\|} 

\newcommand{\less}{\lesssim}
\newcommand{\more}{\gtrsim}
\newcommand{\ep}{\epsilon}
\newcommand{\om}{\omega}
\newcommand{\ga}{\gamma}

\newcommand{\dist}{{\rm dist}}
\newcommand{\const}{{\rm c o n s t  }}

\DeclarePairedDelimiter\floor{\lfloor}{\rfloor}

\makeatletter
\newsavebox\myboxA
\newsavebox\myboxB
\newlength\mylenA

\newcommand*\xoverline[2][0.75]{%
    \sbox{\myboxA}{$\m@th#2$}%
    \setbox\myboxB\null
    \ht\myboxB=\ht\myboxA%
    \dp\myboxB=\dp\myboxA%
    \wd\myboxB=#1\wd\myboxA
    \sbox\myboxB{$\m@th\overline{\copy\myboxB}$}
    \setlength\mylenA{\the\wd\myboxA}
    \addtolength\mylenA{-\the\wd\myboxB}%
    \ifdim\wd\myboxB<\wd\myboxA%
       \rlap{\hskip 0.5\mylenA\usebox\myboxB}{\usebox\myboxA}%
    \else
        \hskip -0.5\mylenA\rlap{\usebox\myboxA}{\hskip 0.5\mylenA\usebox\myboxB}%
    \fi}
\makeatother

\newcommand{\lbar}{\xoverline{l}}

\theoremstyle{plain}
\newtheorem{theorem}{Theorem}
\newtheorem{proposition}{Proposition}

\newtheorem{lemma}{Lemma}

\numberwithin{equation}{section}

\theoremstyle{remark}
\newtheorem{remark}{Remark}

\theoremstyle{definition}
\newtheorem{definition}{Definition}

\title[Birkhoff means of uniquely ergodic toral maps]{Uniform convergence rate for Birkhoff means of certain uniquely ergodic toral maps}

\date{}

\begin{document}

\author[S. Klein]{Silvius Klein}
\address[S. Klein]{Departamento de Matem\'atica, Pontif\'icia Universidade Cat\'olica do Rio de Janeiro (PUC-Rio), Brazil}
\email{silviusk@impa.br}

\author[X. Liu]{Xiao-Chuan Liu}
\address[X. Liu]{Instituto de Matem\' atica e Estat\' istica da Universidade de S\~ao Paulo,
R. do Mat\~ ao, 1010 - Vila Universitaria, S\~ ao Paulo, Brazil} 
\email{lxc1984@gmail.com}

\author[A. Melo]{Aline Melo}
\address[A. Melo]{Departamento de Matem\'atica, Pontif\'icia Universidade Cat\'olica do Rio de Janeiro (PUC-Rio), Brazil}
\email{alinedemelo.m@gmail.com}

\begin{abstract}
We obtain estimates on the uniform convergence rate of the Birkhoff average of a continuous observable over torus translations and affine skew product toral transformations.
The convergence rate depends explicitly on the modulus of continuity of the observable and on the arithmetic properties of the frequency defining the transformation. Furthermore, we show that for the one dimensional torus translation, these estimates are nearly optimal.
\end{abstract}
 
\maketitle

\section{Introduction and statements}\label{introduction}
\input{introduction.tex}

\section{A case study: the one dimensional torus translation} \label{dimension one}
\input{translation-1.tex}

\section{The higher dimensional torus translation case}\label{higher dimension}
\input{translation-d.tex}

\section{The affine skew product case}\label{affine skew product}
\input{skew-product.tex}

\medskip

\begin{remark} Unlike in the case of the one dimensional torus translation, the rate of convergence $\bigo \left(\frac{1}{N^\beta}\right)$ of the Birkhoff means for the other two models considered in this paper are probably far from optimal. 

It would be interesting to see if using multidimensional continued fractions could improve (or indeed, optimize) the estimates for the higher dimensional torus translation model. 

In the affine skew product case, the estimates could be slightly improved in high dimensions ($d > 10$), by combining Weyl's differencing method of bounding exponential sums with that of Vinogradov.  It would also be interesting to obtain some (any) lower bounds on the rate of convergence for this model. 

Finally, it is evident from the proofs that other types of transformations can be analyzed using the same scheme. It would thus be interesting to obtain a more general framework for the applicability of this method.
\end{remark}

\medskip

\subsection*{Acknowledgments} The first two authors are grateful to Qi Zhou for his suggestions, which helped improve the final version of the manuscript.

S. Klein has been supported in part by the CNPq research grant 306369/2017-6 and by a research grant from his institution (PUC-Rio).  

X-C. Liu has been supported by the FAPESP postdoctoral grant 2018/03762-2. He would also like to thank Lorenzo D\'iaz for the hospitality during his visit to PUC-Rio in October 2018, when he joined this project. 

A. Melo has been supported by CAPES and CNPq postgraduate grants.

\bigskip

\providecommand{\bysame}{\leavevmode\hbox to3em{\hrulefill}\thinspace}
\providecommand{\MR}{\relax\ifhmode\unskip\space\fi MR }
\providecommand{\MRhref}[2]{%
  \href{http://www.ams.org/mathscinet-getitem?mr=#1}{#2}
}
\providecommand{\href}[2]{#2}

\end{document}

%% file: introduction.tex
Consider a {\em uniquely ergodic} and minimal dynamical system $(X, \mu, T)$. That is, $X$ is a compact metric space, $\mu$ is a Borel probability measure on $X$ and $T\colon X \to X$ is a continuous transformation. 

Let $\obs \colon X \to \R$ be a {\em continuous} observable and for every integer $N$ let
$$\bsum{N} (x) := \obs (x) + \obs (T x) + \ldots + \obs (T^{N-1} x) $$
denote the corresponding Birkhoff sum. 

Birkhoff's ergodic theorem in this setting asserts that {\em uniformly},
$$\frac{1}{N} \, \bsum{N}  \to \int_X \obs \, d \mu \quad  \text{ when } N \to \infty .$$ 

It is then natural to ask whether for certain more specific systems there is an {\em effective} rate of uniform convergence of the Birkhoff averages, one which depends explicitly on the data.

Standard examples of uniquely ergodic dynamical systems are the torus translation by an incommensurable frequency and  the affine skew product toral transformation. These are precisely the systems considered in this paper.  Let us introduce them more formally.

\medskip

Let $\mathbb{T}=\mathbb{R}/\mathbb{Z}$ be the one dimensional torus endowed with the Lebesgue measure. Let $\om \in (0, 1) \setminus \Q$ be a frequency and define $T_\om \colon \T\to \T$,  $T_\om x = x + \om$
to be the corresponding translation. 
Moreover, a continuous observable $\obs \colon \T \to \R$ may also be regarded as a $1$-periodic function on $\R$ or as a function on $[0, 1]$ with $\obs (0) = \obs (1)$. 

\medskip
 
Let $\T^d = (\R/\Z)^d$ with $d\ge1$ be the $d$-dimensional torus endowed with the Haar measure and also regarded as an additive compact group. 
Points on $\T^d$ are written as $\x = (x_1, \ldots, x_d)$ and we choose on $\T^d$  the distance given by the norm $\abs{\x} := \max_{1\le j \le d} \abs{x_j}$. A continuous observable  $\obs \colon \T^d \to \R$ may also be regarded as a $1$-periodic function in each variable.

Given a frequency vector $\omb = (\om_1, \ldots, \om_d) \in \T^d$, define the torus translation 
$$T_{\omb} \colon \T^d \to \T^d \quad  T_{\omb} \x := \x + \omb \, .$$

If $1, \om_1, \ldots, \om_d$ are linearly independent over $\Q$, that is, if for all nonzero $\kb \in \Z^d$, $\kb \cdot \omb \neq 0 \mod 1$,  then $T_{\omb}$ is an ergodic (in fact, a uniquely ergodic) transformation. 

The second type of toral transformation considered in this paper is the affine skew product.\footnote{In other contexts this transformation is referred to as the skew-shift.} Given an irrational frequency $\om\in\T$, define 
$$S_\om \colon \T^2 \to \T^2, \quad S_\om (x_1, x_2) = (x_1+ x_2, x_2  + \om) \, .$$

More generally, for any dimension $d\ge 2$, let
$$S_\om \colon \T^d \to \T^d, \quad 
S_\om (x_1, x_2, \ldots, x_d) = (x_1+ x_2, x_2  + x_3, \ldots, x_d + \om) \, .$$

The map $S_\om$ is also an example of a uniquely ergodic transformation (see~\cite{Furstenberg-SE}).

\medskip

Given a uniquely ergodic and minimal dynamical system $(X, \mu, T)$ and a conti\-nuous observable $\obs$,  let
$$\rho_N := \norm{ \frac{1}{N} \, \bsum{N} - \int_X \obs \, d \mu}_\infty = \smallo(1) \quad \text{as } \ N \to \infty$$
be the convergence rate of the Birkhoff means of $\obs$. 
Thus the goal is to make the statement $\rho_N = \smallo(1)$ more explicit in the setting of the examples described above.

The ideal scenario is when $ \obs - \int_X \obs $ is a {\em co-boundary}, that is, when there exists a continuous function $\psi \colon X \to \R$  satisfying the co-homological equation:  
$$\obs (x) - \int_X \obs = \psi \circ T \, (x) - \psi  (x) \quad \text{for all } x \in X \, .$$

Then clearly
$$\frac{1}{N} \, \bsum{N} (x) - \int_X \obs = \frac{\psi (T^n x) - \psi (x)}{N} \, ,$$
showing  that in this case,
$$\rho_N = \bigo\left(\frac{1}{N}\right) \, .$$ 

Using Gottschalk-Hedlund theorem one can see that this rate of convergence  is in fact optimal.

\medskip

In the case of the one-dimensional torus translation, a simple Fourier analysis argument shows that for the co-homological equation to have a solution, a certain amount of regularity of the observable $\obs$ and good arithmetic properties of the frequency $\om$ are both {\em necessary} conditions.  Furthermore, by means of multipliers on Fourier coefficients, conditions of this kind  can be shown to be {\em sufficient}, as explained below.

\begin{definition}\label{sDC intro}  If $t\in\R$, let $\norm{t} := \dist (t, \Z)$ denote its distance to the nearest integer. We say that a frequency $\om \in \T$ satisfies a Diophantine condition if there exists $\gamma > 0$ such that for all $k \in \Z \setminus \{0\}$ we have
\begin{equation} \label{sDC}
\norm{ k\om}   \ge \frac{\gamma}{\abs{k} \, \log^2 \left( \abs{k} + 1\right)} \, .
\end{equation}
\end{definition}

Denote by $\sDC_\ga$ the set of all frequencies $\om \in \T$ satisfying~\eqref{sDC} and note that its measure is $1 - \bigo (\ga)$. Hence the set of frequencies satisfying a Diophantine condition (for some $\ga>0$) has full measure. 

\medskip

Assume that the observable $\obs \in C^{1+\alpha} (\T)$ (that is, $\obs$ is differentiable and its derivative is $\alpha$-H\"{o}lder continuous) and 
 that the frequency $\om \in \T$ satisfies the Diophantine condition~\eqref{sDC}. Then the corresponding co-homological equation has a solution (see Section I.8.10 in~\cite{Katznelson}).\footnote{We note that some arithmetic condition in the spirit of~\eqref{sDC} is indeed necessary, as for any Liouville frequency $\om$ there is a smooth observable $\obs$ for which there is no measurable solution to the corresponding co-homological equation.}
Consequently, the rate of convergence $\{\rho_N\}_{N\ge1}$ of the Birkhoff averages over $T_\om$ is optimal, namely of order $\frac{1}{N}$.


The goal is then to address the question of the  convergence rate $\{\rho_N\}_{N\ge1}$ of the Birkhoff means in the {\em low regularity} setting, for instance when the observable $\obs$ is only H\"{o}lder continuous (or in fact when $\obs$ has any other modulus of continuity).  

\medskip

Let us recall these notions in the general setting of a compact metric space $(X, d)$. 
Given $\alpha \in (0, 1]$, 
a function $\obs \colon X\to\R$ is $\alpha$-H\"older continuous if there exists $C < \infty$ such that
$$\abs{ \obs (x) - \obs (y) } \le C d (x, y)^{\alpha} \quad \text{for all } \ x, y\in X .$$

Let $C^\alpha(X)$ be the Banach space of $\alpha$-H\"older continuous functions on $X$, endowed with the usual H\"older norm
$$\norm{\obs}_\alpha := \norm{\obs}_\infty + \sup_{x\neq y} \, \frac{ \abs{ \obs (x) - \obs (y) }  }{d (x, y)^{\alpha} } \, .$$

A map
 $\moc \colon [0,+\infty)\to [0, +\infty)$ is called\footnote{There are other, more general notions of modulus of continuity; however, in our setting, this definition is not restrictive.} a modulus of continuity if it is strictly increasing, sub-additive, continuous and $\moc (0) =0$.
 
Then a continuous  function $\obs \colon X\to\R$ is said to have modulus of continuity $\moc$ if  for some constant $C < \infty$,
 $$\abs{\obs(x)-\obs(y)}\le C \, \moc(d (x, y)) \quad \text{for all } \ x, y\in X .$$

For example, the modulus of continuity $\moc (h)= h^\alpha = e^{ - \alpha\, \log \frac{1}{h}}$, where $0<\alpha\leq 1$, defines $\alpha$-H\"older continuity; $\moc (h)=  e^{- \alpha\, (\log \frac{1}{h})^\kappa}$ for some constants $\alpha, \kappa \in (0, 1]$  defines weak-H\"older continuity, as $\kappa=1$ corresponds to H\"older continuity; $\moc (h)=  (\log \frac{1}{h})^{-1} $ defines $\log$-H\"older continuity.

Any continuous (hence uniformly continuous) function $\obs \colon X \to \R$ has a modulus of continuity, namely the function
$$\moc_\obs (h) := \sup \left\{ \abs{ \obs(x)-\obs(y)} \colon x, y \in X \ \text{ with } \ d (x, y) \le h \right\} \, .$$

Given a modulus of continuity $\moc$, let $C^\moc (X)$ be the Banach space of continuous functions on $X$ endowed with the norm
$$\norm{\obs}_\alpha := \norm{\obs}_\infty + \sup_{x\neq y} \, \frac{ \abs{ \obs (x) - \obs (y) }  }{ \moc (d (x, y) ) } \, .$$

\medskip

Returning to the one-dimensional torus translation by a frequency $\om$, we recall Denjoy-Koksma's inequality, proven by M. Herman (see~\cite[Chapter VI.3]{Herman_conjugaison}). 
Let $\obs \colon \T \to \R$ be a function of bounded variation and let $\frac{p}{q}$ be such that $\abs{\om - \frac{p}{q}} \le \frac{1}{q^2}$ and $\gcd (p, q) = 1$ (for instance, $\frac{p}{q}$ is a convergent of $\om$). Then
$$\rho_q = \norm{ \frac{1}{q} \, \bsum{q}  - \int_\T \obs}_\infty \le \frac{\text{ var } {(\obs)}}{q} \, .$$

The same type of result also holds for H\"{o}lder continuous observables (or indeed for observables with any modulus of continuity): given $\obs \in C^\alpha (\T)$, where $0<\alpha\le 1$, for every $q$ as above we have:
\begin{equation}\label{D-K ineq}
\rho_q = \norm{ \frac{1}{q} \, \bsum{q}  - \int_\T \obs}_\infty \le \frac{\norm{\obs}_\alpha}{q^\alpha} \, .
\end{equation}

The Ostrowski numeration says that any integer $N$ can be represented as $$N = b_0 + b_1 q_1 + \ldots + b_s q_s + b_{s+1} q_{s+1} \, ,$$ 
where $\{q_n\}_{n\ge1}$ are the denominators of the principal convergents of $\om$ and the coefficients $b_n$ are positive integers such that $b_n < \frac{q_{n+1}}{q_n}$ for all $0\le n\le s+1$. 

Diophantine conditions like~\eqref{sDC} can be characterized by a bound on the lacuna\-rity of the sequence $\{q_n\}_{n\ge1}$. Thus Ostrowski's numeration can be used to derive extensions of Denjoy-Koksma's inequality to {\em all } Birkhoff sums (see Proposition 4.2.8 in~\cite{Queffelec}). That is, the uniform convergence rate $\{\rho_N\}_{N\ge1}$ of the Birkhoff averages over the torus translation $T_\om$ of an observable $\obs$ can be described explicitly in terms of the modulus of continuity of $\obs$ and the arithmetic properties of $\om$. 

\medskip

In this paper we present a new proof (and a slightly sharper version) of the aforementioned extension of Denjoy-Koksma's inequality. Our approach uses Fourier analysis tools (effective approximation by trigonometric polynomials along with an effective rate of decay of the Fourier coefficients of continuous functions) and estimates of (in this case, some simple) exponential sums. 

\begin{theorem} \label{extension D-K} 
Assume that the observable $\obs$ is an $\alpha$-H\"{o}lder continuous function on $\T$ and 
 that the frequency $\om \in \T$ satisfies the Diophantine condition~\eqref{sDC}.
Then for all integers $N$ we have\footnote{In the statements of the main results, and indeed  throughout the whole paper, $\const$ stands for a {\em universal} constant, one which could (but will not) be specified, and which may change from an estimate to another. Moreover, for two {\em varying} quantities $a$ and $b$, $a \less b$ will mean that $a \le \const \, b$, where $\const$ is such a universal constant, while $a \asymp b$ will mean that  $a \less b$ and $b \less a$. Furthermore, $a \less_d b$ will mean that $a \le C b$, where the constant $C$, rather than simply being universal, only depends on the parameter $d$.
}
\begin{equation}\label{extension D-K eq}
 \norm{ \frac{1}{N} \, \bsum{N}  - \int_\T \obs}_\infty \le \const \, \left(\frac{1}{\ga} \log  \frac{1}{\ga}\right) \, \norm{\obs}_\alpha \, \frac{\log^{3 \alpha} N}{N^\alpha} \, .
 \end{equation}
 \end{theorem}

\begin{remark}
A similar result holds for observables with any modulus of continuity $\moc$. Thus instead of the rate of convergence of order (nearly) $\frac{1}{N^\alpha}$, valid for $\alpha$-H\"older observables, one has a rate of order (nearly) $w \left( \frac{1}{N}\right)$.
Similar results also hold for frequencies satisfying much more general arithmetic conditions.
\end{remark}

The general Fourier analysis approach which we employ to establish Theorem~\ref{extension D-K} allows us to prove that the Denjoy-Koksma inequality is optimal; in particular, the convergence rate~\eqref{extension D-K eq} is also (nearly) optimal.

\begin{theorem}\label{limitations intro}
Let $\alpha \in (0, 1]$. For almost every frequency $\om\in\T$ (thus for a.e. Diophantine frequency) there exists an observable $\obs\in C^\alpha(\T)$ and a subsequence $\{q_{N_k}\}_{k\ge1}$ of denominators of principal convergents of $\om$ such that the corresponding Birkhoff averages of $\obs$ over the transformation $T_\om$ satisfy
\begin{equation*}
\abs{\frac{1}{q_{N_k}} \obs^{(q_{N_k})} (0) - \int_\T \obs } \ge \const \frac{1}{q_{N_k}^\alpha} \quad \text{for all } \ k\ge1 .
\end{equation*}
\end{theorem}

\begin{remark}
A similar result holds for any other modulus of continuity $\moc$: there exists an observable $\obs\in C^\moc (\T)$ such that the rate of convergence of the Birkhoff means along a subsequence of denominators of convergents of $\om$ is slower than $\moc \left(\frac{1}{N}\right)$. 

Moreover, our construction also shows the (near) optimality of the convergence rate $\{\rho_N\}_{N\ge1}$ for frequencies satisfying more general arithmetic conditions. 

Furthermore, this construction can be tweaked to show the (probably known) fact that given any (arbitrarily slow) rate of convergence to zero $\{\rho_N\}_{N\ge1}$, for an appropriate Liouvillean frequency $\om$, there is an {\em analytic} observable $\obs$ so that the convergence rate of the corresponding Birkhoff averages is slower than $\{\rho_N\}_{N\ge1}$.
\end{remark}

Returning to the extension of Denjoy-Koksma's inequality in Theorem~\ref{extension D-K}, we note that the estimate~\eqref{extension D-K eq}  is the prototype of what we aim for regarding the other dynamical systems considered (the higher dimensional torus translation and the affine skew product):  the dependence on the data (frequency and observable) is explicit and stable under appropriate perturbations.

J.-C. Yoccoz  (see~\cite[Appendice 1]{Yoccoz}) constructed an example showing that there is {\em no analogue} of the Denjoy-Koksma inequality for the two dimensional torus translation. 
More precisely, he showed that given any (arbitrarily slow) rate of convergence to zero $\{\rho_N\}_{N\ge1}$, there are an {\em analytic} observable $\obs \colon \T^2 \to \R$ and a rationally independent frequency $\omb \in\T^2$ such that for almost every $\x \in \T^2$,  and for $N$ large enough,
$$ \abs{\frac{1}{N} \, \bsum{N} (\x) - \int_{\T^2} \obs} \ge \rho_N \, .$$

However, the frequency vector $\omb=(\om_1, \om_2)$ was constructed in such a way that its components $\om_1$ and $\om_2$ are (highly) Liouville. The question is then what happens for {\em typical} (in a measure theoretical sense) frequency vectors.  

Again, as in the one-dimensional case, by means of Fourier analysis arguments, the co-homological equation has a solution provided that the observable be regular enough\footnote{In the $d$-dimensional setting, more regularity is needed, namely at least that $\obs \in C^{d+\alpha} (\T^d)$.}  and the frequency satisfy a Diophantine condition. 
Our goal is then to study the low regularity setting.

\begin{definition}\label{DC intro} For a multi-index $\kb = (k_1, \ldots, k_d) \in \Z^d$, let $\abs{\kb} := \max_{1\le j \le d} \, \abs{k_j}$, and if $\x = (x_1, \ldots, x_d) \in \T^d$, let
$\kb \cdot \x := k_1 x_1 + \ldots + k_d x_d$. 

We say that $\omb \in \T^d$ satisfies a Diophantine condition if there exist $\ga>0$ and $A>d$ such that 
\begin{equation}\label{DC}
\norm {\kb \cdot \omb} = \dist \left( \kb \cdot \omb, \,  \Z \right) \ge \frac{\ga}{\abs{\kb}^A}
\end{equation}
for all  $\kb \in \Z^d$ with  $\abs{\kb} \neq 0$.
Denote by $\DC_{\ga, A}$ the set of all frequency vectors $\omb$ satisfying the Diophantine condition~\eqref{DC}. Given any $A>d$, the set $\bigcup_{\ga>0} \DC_{\ga, A}$  has full measure. 
\end{definition}

Next we formulate our results on the rate of convergence of the Birkhoff means of H\"older\footnote{It will be evident from their proofs that similar kinds of estimates can be derived for observa\-bles with any modulus of continuity; however, for simplicity, we will not formulate them.} observables over a higher dimensional  toral translation and over an affine skew product.

\begin{theorem}\label{thm d translation intro}
Let $\obs \in C^\alpha (\T^d)$, let $\omb \in \DC_{\ga, A}$ and  let $T_{\omb} \colon \T^d \to \T^d$ be the corresponding torus translation. Then for all $N \ge 1$ we have
$$\norm{\frac{1}{N} \, \bsum{N}  - \int_{\T^d} \obs}_\infty \le \frac{\const}{\ga} \, \norm{\obs}_\alpha \, \frac{1}{N^\beta} \, $$
where $\beta = \frac{\alpha}{A+d}$ .
 \end{theorem}
 
\begin{theorem}\label{skew product thm intro}
Let $\obs \in C^\alpha (\T^d)$, let $\om \in \sDC_\ga$ and let $S_\om \colon \T^d \to \T^d$ be the corresponding affine skew product transformation.
Then for all $N\ge 1$ we have
$$\norm{\frac{1}{N} \bsum{N} - \int_{\T^d} \obs }_\infty \le \const_{d, \ga} \, \norm{\obs}_\alpha \, \frac{1}{N^\beta} \, ,$$
where $\beta = \frac{\alpha \delta}{\delta+d} - \smallo(1)$ and $\delta = \frac{1}{2^{d-1}}$ .
\end{theorem}

\medskip

The rest of the paper is organized as follows. In Section~\ref{dimension one} we present the proof of Theorem~\ref{extension D-K}, the extension of Denjoy-Koksma's inequality, and construct an example showing its (near) optimality, thus proving Theorem~\ref{limitations intro}. Moreover, we discuss various other similar results for more general arithmetic conditions.  
While more technical than the proof of M. Herman, our argument for establishing the Denjoy-Koksma type inequality~\eqref{extension D-K eq}  is modular and versatile; as such, it provides a simple case study for the proofs of Theorem~\ref{thm d translation intro} and Theorem~\ref{skew product thm intro}, which are presented in the subsequent two sections and follow the same scheme.

%% file: translation-1.tex
We obtain\footnote{Our method provides a slightly stronger  bound than the approach used to prove Proposition 4.2.8 in~\cite{Queffelec}.} a different proof of the extension~\eqref{extension D-K eq} of Denjoy-Koksma's inequality. Our approach can be described as  a quantitative version of the proof of Weyl's equidistribution theorem.


\subsection*{Some arithmetic considerations}

We begin with a review of some basic arithmetic properties that will be used in the sequel. The interested reader will find all necessary details in~ \cite[Chapter 3]{Queffelec} or in~\cite{Lang}.

It is easy to verify that for a real number $t$ and a positive integer $n$,
\begin{equation}\label{norm nt}
\abs{e^{2 \pi i t} - 1} \asymp \norm{t}, \quad \text{so} \quad \norm{n t} \less n \norm{t} .
\end{equation}

Given $\omega \in (0,1) \setminus \Q$, consider its continued fraction expansion 
$$\omega = [a_1, a_2, \ldots, a_n, \ldots ]$$
and for each $n\ge1$, the corresponding $n$-th (principal) convergent
$$\frac{p_n}{q_n} = [a_1, \ldots, a_n] \, ,$$
where the integers $p_n = p_n (\omega)$, $q_n = q_n (\omega)$ are relatively prime.

Setting $p_0=0$ and $q_0=1$, the convergents satisfy the recurrence relations 
$p_{n+1} = a_{n+1} p_n + p_{n-1}$ and $q_{n+1} = a_{n+1} q_n + q_{n-1}$ for all $ n\ge1$.

Thus the sequence $\{q_n\}_{n\ge 1}$ of denominators of the convergents of $\omega$ is strictly increasing, $\frac{q_{n+1}}{q_n} \asymp a_n$ and $q_n \ge 2^{(n-1)/2}$ for all $n\ge1$. 

Moreover,
$q_{n+2} = a_{n+1} q_{n+1} + q_n \ge 2 q_n$, so $q_{n+2 k} \ge 2^k q_{n}$ for all integers $n, k$.

Furthermore, $\norm{q_n \om} = \abs{q_n \om - p_n}$ and the following inequalities hold: 
\begin{equation}\label{q_n-ineq1}
\frac{1}{2q_{n+1}} < \abs{q_n\omega-p_n} < \frac{1}{q_{n+1}} \, .
\end{equation}

We call a {\em best approximation} to $\om$ any (reduced) fraction $\frac{p}{q}$ such that $\normT{q \om} = \abs{q \om - p}$ and
$$\normT{j \om} > \normT{q \om} \quad \text{ for all } \quad 1 \le j < q \, .$$

It turns out that the best approximations to $\om$ are precisely its principal convergents. In fact,  $q_{n+1}$ is the smallest integer $j > q_{n}$ such that $\normT{j \om} < \normT{q_{n} \om}$ (see~\cite[Chapter 1, Theorem 6]{Lang}). 
Noting also that $\norm{-t}=\norm{t}$, we conclude the following.

Given $\omega\in (0, 1) \setminus \Q$, if $\frac{p}{q}$ is a best approximation to $\om$ then 
\begin{equation}\label{ineq2}
\normT{j \om}  >  \frac{1}{2 q} \quad \text{ for all } \quad 1 \le \abs{j} < q \, .
\end{equation}

The Diophantine condition~\eqref{sDC} on $\om \in \T$ (or any other such arithmetic condition) can be described in terms of the relative sizes of the gaps between the denominators $q_n$. 

More precisely, if $\om \in \sDC_\ga$, since by~\eqref{q_n-ineq1} we have
$\normT{q_n \om}  = \abs{q_n\omega-p_n} < \frac{1}{q_{n+1}}$, it follows that
\begin{equation}\label{lacunarity q_n}
q_{n+1} \le \frac{1}{\ga} \, q_n \log^2 (q_n + 1)  \quad \text{ for all } n \ge 1 .
\end{equation}

Thus if $N\ge 2$, choosing $s$ to be the largest integer such that $q_s \le N$, we have that
$$q_s \le N  < q_{s+1} \le \frac{1}{\ga} \, q_s \log^2 (q_s+1)   \le \frac{1}{\ga} \, N \log^2 (N+1) \, .$$

We then conclude the following: given $\ga>0$ and  $\om\in\sDC_\ga$, for all $N\ge2$ there is a principal convergent $\frac{p}{q}$ of $\om$ such that
\begin{equation}\label{arithmetic star}
N \le q \le \frac{1}{\ga} \, N \log^2 (N+1) \, .
\end{equation}  

Similarly, given $\ga>0$, $A>1$ and $\om \in \sDC_{\ga, A}$ (that is, $\om \in \T$ satisfies the slightly more general Diophantine condition~\eqref{DC}, $d=1$), the lacunarity of the corresponding sequence $\{q_n\}$ has the bound
\begin{equation}\label{lacunarity DC}
q_{n+1} \le \frac{1}{\ga} \, q_n^A   \quad \text{ for all } n \ge 1 .
\end{equation}

Finally, let us define a more general arithmetic condition, one that includes some Liouville numbers. Given $\om \in (0, 1) \setminus \Q$ and $\{q_n\}$ the denominators of its convergents, let
\begin{equation*}
\beta (\om) := \limsup_{n \to \infty} \frac{\log q_{n+1}}{q_n} \, .
\end{equation*}

Note that for frequencies satisfying a Diophantine condition $\om\in\sDC_{\ga, A}$ we have $\beta (\om) = 0$. 

We will consider frequencies $\om$ with $\beta (\om) < \infty$. If $\beta (\om) < \beta  <\infty$, then for some $C = C(\om) < \infty$ we have
\begin{equation}\label{lacunarity beta}
q_{n+1} \le C \, e^{\beta q_n} \, \text{ for all } n \ge 1 \, .
\end{equation}


\subsection*{Effective approximation by trigonometric polynomials} We review some Fourier analysis notions used in our proof of~\eqref{extension D-K eq} (see~\cite{Stein-Princeton-lectures} or~\cite{Katznelson} for more details).

Given a {\em continuous} observable $\obs \colon \T \to \R$ and $k \in \Z$, let 
\begin{equation*}
\fc{\obs}(k) = \int_0^1 \phi (x) e  (- k x) dx
\end{equation*}
be its $k$-th Fourier coefficient, where we use Vinogradov's notation $e (x) := e^{2 \pi i x}$.

It is known that if $\obs\in\C^\alpha(\T)$, then the Fourier coefficients of $\obs$ have the decay 
\begin{equation}\label{FC decay dim 1}
\abs{\fc{\obs} (k)} \le \const \, \norm{\obs}_\alpha \, \frac{1}{\abs{k}^\alpha} \quad \text{for all } \ k\neq 0 .
\end{equation}

More generally, given a modulus of continuity $\moc$, if $\obs\in C^\moc (\T)$, then
\begin{equation*}
\abs{\fc{\obs} (k)} \le \const \, \norm{\obs}_\moc \, \moc \left( \frac{1}{\abs{k}}\right) \quad \text{for all } \ k\neq 0 .
\end{equation*}

These estimates will play a crucial r\^ole in our proof of the Denjoy-Koksma type inequality.

Consider the Fourier series associated with $\obs$
$$\obs (x)  \sim  \sum_{k= - \infty}^{\infty} \, \fc{\obs} (k) \, e (k x) = \int_\T \obs + \sum_{k\neq0}  \, \fc{\obs} (k) \, e (k x) \, .$$

For every $n\ge 0$, denote by
$$S_n \obs (x) := \sum_{\abs{k}\le n}  \, \fc{\obs} (k) \, e (k x) = (\obs \ast D_n) (x)\, $$
the $n$-th partial sum\footnote{For a given observable $\obs$ and an integer $n$, we consider its $n$-th Birkhoff sum $\bsum{n}$ (over some transformation $T$) and the $n$-th partial sum $S_n \obs$ of its Fourier series; we also define a trigonometric polynomial $\obs_n$ that approximates $\obs$. The reader should mind the difference between all of these notations.} of the Fourier series of $\obs$, where
$$D_n (x) = \sum_{\abs{k}\le n} e (k x)$$
 is the $n$-th Dirichlet kernel.

The partial sums $S_n \obs (x)$ do {\em not}, in general, converge uniformly (or even pointwise, for all points) to $\obs (x)$. 
However, by Fej\'{e}r's theorem, their Ces\`{a}ro averages do converge uniformly, and at a rate that depends explicitly on the modulus of continuity of the function $\obs$.  
More precisely, let
$$ \frac{S_0 \obs (x) + \ldots + S_{n-1} \obs (x)}{n}  = (\obs \ast F_n) (x) \, $$
be the Ces\`{a}ro means of the Fourier series of $\obs$, 
where 
$$F_n (x) = \frac{D_0 (x) + \ldots + D_{n-1} (x)}{n} = \sum_{\abs{k}\le n}  \left( 1 - \frac{\abs{k}}{n} \right) \, e (k x) =  \frac{1}{n} \, \frac{\sin^2 (n x / 2)}{\sin^2 (x / 2)}$$ 
is the $n$-th Fej\'{e}r kernel.  

If $0<\alpha\le 1$ and $\obs \in C^\alpha (\T)$, then for all $n\ge1$
$$\norm{\obs - \obs \ast F_n}_\infty   \le \const \, \norm{\obs}_\alpha \,  \frac{\log n}{n^\alpha} \, .$$

This is a quantitative version of the Weierstrass approximation theorem. Since $F_n (x)$ is a trigonometric polynomial of degree $\le n$, so is $\obs \ast F_n$. Thus $\alpha$-H\"{o}lder continuous functions can be approximated by trigonometric polynomials of degree $\le n$ with error bound of order  $\frac{\log n}{n^\alpha}$.

This approximation is not optimal. The convolution with the $n$-th Jackson kernel (instead of the Fej\'{e}r kernel) provides an error bound of order $\frac{1}{n^\alpha}$, which is optimal (this result  is called Jackson's theorem, see~\cite{Katznelson} or~\cite{Shadrin}).\footnote{However, using the (better known) Fej\'{e}r summability kernel here would still provide good (albeit slightly weaker) estimates.}  Let us introduce this summability kernel (see~\cite{Shadrin} for more details).

For all $n\in\N$, let $J_n \colon \T \to \R$, $$J_n (x):= c_n \, F_m^2 (x) \, $$ where $m=\floor{\frac{n}{2}}$, $F_m(x)$ is the Fej\'er kernel and $c_n \asymp \frac{1}{n}$ is a normalizing factor chosen so that $\int_\T J_n = 1$. Therefore, $J_n(x)$ is a trigonometric polynomial of degree $\le n$ and uniformly in $n$ and $k$, its (Fourier) coefficients have the bound $\abs{\fc{J_n}(k)} \less 1$. 

Moreover, $\{J_n\}_{n\ge1}$ is a summability kernel and in fact, given $\obs\in C^\alpha(\T)$ with $\alpha \in (0, 1]$, if we denote
$$\obs_n (x) := (\obs \ast J_n) (x) \, ,$$
then $\obs_n$ is a trigonometric polynomial of degree $\le n$ and 
\begin{equation}\label{Jackson Holder 1d}
\norm{\obs - \obs_n}_\infty \le \const \, \norm{\obs}_\alpha \,  \frac{1}{n^\alpha} \, .
\end{equation}

A similar result holds for continuous functions on $\T$ with any (sub-additive) modulus of continuity $\moc$. If $\obs \in C^\moc (\T)$ then
\begin{equation*}\label{Jackson type d1}
\norm{\obs - \obs_n}_\infty \le \const \, \norm{\obs}_{\moc} \,  \moc \left( \frac{1}{n} \right) \, .
\end{equation*}

\subsection*{Birkhoff averages of Fourier modes}
Consider the Fourier series associated with the continuous observable $\obs$:
$$\obs (x)  \sim  \sum_{k= - \infty}^{\infty} \, \fc{\obs} (k) \, e (k x) = \int_\T \obs + \sum_{k\neq0}  \, \fc{\obs} (k) \, e (k x) \, .$$

Then for all $j \in \Z$,
$$\obs (x + j \om) -  \int_\T \obs \sim \sum_{k\neq0}  \, \fc{\obs} (k) \, e (k (x +j \om))  =  \sum_{k\neq0}  \, \fc{\obs} (k) \, e (j k \om) \, e (k x) \, .$$ 

It follows that the Fourier series of the $N$-th Birkhoff average of $\phi$ is
\begin{align}\label{Birkhoff sum kernel}
\frac{1}{N} \bsum{N} (x)  - \int_\T \obs \sim 
& \sum_{k\neq0}  \, \fc{\obs} (k) \, \left(  \frac{1}{N} \sum_{j=0}^{N-1}  e (j k \om)  \right)  \, e (k x) \, \notag \\
=  &\sum_{k\neq0}  \, \fc{\obs} (k) \, \esum_N (k \om)  \, e (k x) \, , 
\end{align}
where $\esum_N$ refers to the (averaged) exponential sum
$$\esum_N (t):= \frac{1}{N}\sum_{j=0}^{N-1}  e (j t) .$$ 
 
Clearly $\abs{ \esum_N (t) } \le 1$. Moreover, since $\esum_N (t)$ is the sum of a finite geometric sequence, we have 
\begin{equation}\label{exp sum 1}
\abs{\esum_N (t)}  
  = \frac{1}{N} \abs{ \frac{1- e (N t)}{1-e(t)} }
 \less \frac{1}{ N \normT{t}} \, .
\end{equation}

\begin{lemma} \label{estimate kernel}
Let $\frac{p}{q}$ be a best approximation to the irrational number $\om$. 
Then 
\begin{equation} \label{sum kernel bound }
\sum_{1 \le \abs{k} < q} \abs{\esum_N(k\omega)}  \le \const \, q \,  \frac{\log q}{N} \, .
\end{equation}
\end{lemma}

\begin{proof} Since $\norm{t}=\norm{-t}$ for all $t\in\R$, it is enough to bound $\sum_{1 \le k < q} \abs{\esum_N(k\omega)}$.

Recall that by~\eqref{ineq2}, for all $1\leq \abs{j} <q$,
$$
\normT{ j\om} >\frac{1}{2q} \, .
$$

Thus for every $k, k' \in \{1, \ldots, q-1\}$ with $k\neq k'$ we have
$$
\normT{ k\omega-k'\omega} = \normT{ (k-k')\omega }  > \frac{1}{2q} \, .
$$

Divide $\T$ into the $2q$ arcs $C_j=[\frac{j}{2q}, \frac{j+1}{2q})$, $0\leq j\leq 2q-1$, with equal length $\abs{C_j}=\frac{1}{2q}$. By the bound above,  each arc $C_j$ contains at most one point $k\omega \mod 1$ with $k\in  \{1, \ldots, q-1\}$, and by~\eqref{ineq2}, the arcs $C_0 = [0, \frac{1}{2 q})$ and $C_{2q-1} = [ \frac{2q-1}{2q}, 1)$ do {\em not} contain any such point.

Note that if $x \in C_j$ with $1\leq j\leq q-1$, then $\norm{x}\geq \frac{j}{2q}$ and similarly, if  $x\in C_{2q-j-1}$ with $1\leq j\leq q-1$ then $\norm{x}\geq \frac{j}{2q}$. Using~\eqref{exp sum 1}, it follows that 
\begin{align*}
  \sum_{1 \le k < q} \abs{ \esum_N(k\omega)} \le   \sum_{1 \le k < q} \frac{1}{N \norm{k \om}} 
   \le \frac{1}{N} \sum_{j=1}^{q-1} \frac{4 q}{j } 
\le  \const \,  \frac{q}{N} \log q \, ,
\end{align*}
which proves the lemma.
\end{proof}

%


\begin{proof}[Proof of Theorem~\ref{extension D-K}]
Let $\obs \in C^\alpha (\T)$, fix $N\ge 1$ (the length of the Birkhoff sum) and let $1\le n \le N$ (the degree of polynomial approximation) to be chosen later. 
Write 
$$\obs = \obs_n + \left( \obs -\obs_n \right)  =:  \obs_n + \psi_n \,, $$
which implies that
\begin{align}\label{case study eq0}
\frac{1}{N} \, \bsum{N}  - \int_\T \obs & 
= \left(\frac{1}{N} \, \bsumap{N}   - \int_\T\obs_n \right) + \left(\frac{1}{N} \, \psi_n^{(N)} - \int_\T \psi_n\right) .
\end{align}

From ~\eqref{Jackson Holder 1d} we clearly have
\begin{equation}\label{1d tail est}
\norm{ \frac{1}{N} \, \psi_n^{(N)} - \int_\T \psi_n }_\infty \le \const \, \norm{\obs}_\alpha \,  \frac{1}{n^\alpha} \, .
\end{equation}

The proof is then reduced to estimating the $N$-th Birkhoff average of the trigonometric polynomial (of degree $\le n$) $\obs_n = \obs \ast J_n$. 
By~\eqref{Birkhoff sum kernel} applied to $\obs_n$, for all $x\in\T$ we have:
\begin{align*}
\frac{1}{N} \, \bsumap{N} (x)  - \int_\T\obs_n = \sum_{1\le\abs{k}\le n} \fc{\obs_n} (k) \, \esum_N (k \om) \, e (k x) \, .
\end{align*}

For all $k\neq 0$, the corresponding Fourier coefficient of $\obs_n = J_n \ast \obs $ satisfies
\begin{equation*}
\abs{ \fc{\obs_n} (k)} = \abs{\fc{J_n}(k) \cdot \fc{\obs} (k)} \less  \abs{\fc{\obs} (k)} \less \norm{\obs}_\alpha \, \frac{1}{\abs{k}^\alpha} \, . 
\end{equation*}

Since the sequence $\{q_j\}_{j\ge1}$ of the denominators of the principal convergents of $\om$ is strictly increasing, there is an integer $s$ such that $q_s \le n < q_{s+1}$. 

Thus combining with the preceding, for all $x\in\T$ we have:
\begin{align}\label{main bound D-K}
\abs{ \frac{1}{N} \, \bsumap{N} (x)  - \int_\T\obs_n  } & \le \sum_{1\le\abs{k}\le n} \abs{\fc{\obs_n} (k)} \, \abs{\esum_N (k \om)}    \le  \sum_{1\le\abs{k} < q_{s+1}} \abs{\fc{\obs_n} (k)} \, \abs{\esum_N (k \om)} \notag\\
& \less \norm{\obs}_\alpha \,  \sum_{1\le\abs{k} < q_{s+1}} \frac{1}{\abs{k}^\alpha} \, \abs{\esum_N (k \om)} \notag\\
&  =  \norm{\obs}_\alpha \,   \sum_{j=1}^s \, \sum_{q_j \le \abs{k} < q_{j+1}} \frac{1}{\abs{k}^\alpha} \, \abs{\esum_N (k \om)}  \notag\\
& \le  \norm{\obs}_\alpha \,  \sum_{j=1}^s \frac{1}{q_j^\alpha}  \sum_{1 \le \abs{k} < q_{j+1}}  \abs{\esum_N (k \om)}  \notag\\
& \less    \norm{\obs}_\alpha \, \frac{1}{N} \, \left(   \sum_{j=1}^s \frac{q_{j+1}  \log q_{j+1}}{q_j^\alpha} \right) 
\end{align}
where the last estimate follows from Lemma~\ref{estimate kernel}.

If $\om\in\sDC_\ga$, then using~\eqref{lacunarity q_n}, for all $1\le j \le s$ we have:
\begin{align*} 
\frac{q_{j+1} \log(q_{j+1})}{q_j^\alpha} 
            & \less \frac{\left( \frac{1}{\ga} \,  q_j \log^2 (q_j+1) \right) \, \left( \log \frac{1}{\ga} \, + \log (q_j+1) \right) }{q_j^\alpha} \\
            & \less \left(  \frac{1}{\ga} \, \log \frac{1}{\ga} \right) \, q_j^{1-\alpha} \, \log^3 (q_j+1)  
            \less \left(  \frac{1}{\ga} \, \log \frac{1}{\ga} \right) \, q_j^{1-\alpha} \, \log^3 n \, .
\end{align*}

As mentioned earlier, $q_n \le 2^{- k} q_{n+2k}$ for all $n, k \in \N$, hence
\begin{align*}
\sum_{j=1}^s q_j^{1-\alpha} 
= \sum_{\substack{1\le j\le s\\j \text{ odd }}} q_j^{1-\alpha}  + \sum_{\substack{1\le j\le s\\j \text{ even }}} q_j^{1-\alpha}
 \less q_s^{1-\alpha}  \le n^{1-\alpha} \, . 
\end{align*}

We obtained the following: if $\om \in \sDC_\ga$ and $q_s \le n < q_{s+1}$, then
\begin{equation}\label{sum q's}
 \sum_{j=1}^s \frac{q_{j+1}  \log q_{j+1}}{q_j^\alpha}  \less  \left(  \frac{1}{\ga} \, \log \frac{1}{\ga} \right) \, n^{1-\alpha} \, \log^3 n \, . 
\end{equation}

Combining~\eqref{case study eq0},~\eqref{1d tail est}, ~\eqref{main bound D-K},
and~\eqref{sum q's} we obtain
\begin{align*}
 \norm{ \frac{1}{N} \, \bsum{N}  - \int_\T \obs}_\infty \less \left(\frac{1}{\ga} \log  \frac{1}{\ga}\right) \, \norm{\obs}_\alpha \, 
 \left(  \frac{n^{1-\alpha} \, \log^3 n}{N} + \frac{1}{n^\alpha} \right) \, .
 \end{align*}

The conclusion~\eqref{extension D-K eq} then follows by choosing the degree of polynomial approxi\-mation  $n= \left \lfloor{   \frac{N}{\log^3 N } } \right \rfloor$. 
\end{proof}

 \begin{remark} 
 Similar types of statements hold if instead of the Diophantine condition~\eqref{sDC} we consider more general arithmetic conditions on $\om$. The only part of the argument where such a condition is needed is in the derivation of the bound~\eqref{sum q's}.

\smallskip

$\blob$ If $\om \in \sDC_{\ga, A}$, 
then using~\eqref{lacunarity DC}, the  analogue of the bound~\eqref{sum q's} is:
 $$ \sum_{j=1}^s \frac{q_{j+1}  \log q_{j+1}}{q_j^\alpha}  \le \const_{\ga, A}  \, n^{A-\alpha} \, \log n \, ,$$
 so the rate of convergence of the corresponding Birkhoff means has the bound:
 \begin{equation}\label{bound DC eq}
  \norm{ \frac{1}{N} \, \bsum{N}  - \int_\T \obs}_\infty \le \const_{\ga, A} \, \norm{\obs}_\alpha \,  \frac{\log^{\alpha/A} N}{N^{\alpha/A}} \, .
  \end{equation}

$\blob$ If $\beta (\om) < \beta < \infty$, then using~\eqref{lacunarity beta}, the  analogue of the bound~\eqref{sum q's} is:
$$ \sum_{j=1}^s \frac{q_{j+1}  \log q_{j+1}}{q_j^\alpha}  \le \const_{\om}  \, e^{\beta n} \, n^{1-\alpha} \,  ,$$
so the rate of convergence of the corresponding Birkhoff means has the bound:
 \begin{equation}\label{bound beta eq}
 \norm{ \frac{1}{N} \, \bsum{N}  - \int_\T \obs}_\infty \le \const_{\om} \, \norm{\obs}_\alpha \,  \frac{1}{\log^\alpha N} \, .
 \end{equation}
\end{remark}

\begin{remark}
The case of an observable $\obs$ with an arbitrary modulus of continuity $\moc$ is treated in the same way. The main ingredients of the proof as regards $\obs$ are the order of trigonometric polynomial approximation and the decay of its Fourier coefficients. Recall that if $\obs\in C^\moc(\T)$, then
$\norm{\psi_n}_\infty = \norm{\obs - \obs_n}_\infty \less \norm{\obs}_\moc \, \moc (\frac{1}{n})$ and $| \fc{\obs} (k) |\less \norm{\obs}_\moc \moc (\frac{1}{n})$. Then the estimate~\eqref{main bound D-K} becomes:
$$\abs{ \frac{1}{N} \, \bsumap{N} (x)  - \int_\T\obs_n  } 
 \less   \norm{\obs}_\moc \, \frac{1}{N} \,  \left(   \sum_{j=1}^s q_{j+1} \, \log q_{j+1} \, \moc \Big( \frac{1}{q_j} \Big)  \right) \, .
 $$

 As before, if say $\om \in \sDC_\ga$, then
 $$q_{j+1} \, \log q_{j+1}  \less  \left(\frac{1}{\ga} \log  \frac{1}{\ga}\right) \,  q_j \log^3 (q_j+1)  \less  \left(\frac{1}{\ga} \log  \frac{1}{\ga}\right) \, q_j \, \log^3 n \,.$$
 
Since $\moc$ is sub-additive and increasing, if $q\le n$ then
 $$q \,  \moc  \left(\frac{1}{q} \right) = q \, \moc   \left( \frac{n}{q} \, \frac{1}{n} \right) \le q \,   \left( \frac{n}{q} +1 \right) \,  \moc  \left(\frac{1}{n} \right) 
 \le 2 n \,  \moc  \left(\frac{1}{n} \right)  \, .
 $$
 
 Clearly $s \less \log n$ (since $\{q_j\}$ grows exponentially and $q_s \le n$). Then
 \begin{align*}
  \sum_{j=1}^s q_{j+1} \, \log q_{j+1} \, \moc \Big( \frac{1}{q_j} \Big)  & \less  \left(\frac{1}{\ga} \log  \frac{1}{\ga}\right) \, \log^3 n \, \sum_{j=0}^s  q_{j} \, \moc \Big( \frac{1}{q_j} \Big) \\
& \less  \left(\frac{1}{\ga} \log  \frac{1}{\ga}\right) \,  \left(\log^4 n\right)  \, n \,  \moc  \left(\frac{1}{n} \right)  \, .
 \end{align*}
 
 Putting it all together and choosing $n\asymp \frac{N}{\log^4 N}$ we conclude:
 \begin{align*}
\norm{ \frac{1}{N} \,  \bsum{N} - \int_\T \obs }_\infty & \less \left(\frac{1}{\ga} \log  \frac{1}{\ga}\right) \norm{\obs}_\moc \, \left( n \,  \moc  \Big(\frac{1}{n} \Big) \,  \frac{\log^4 n}{N} +  \moc \Big(\frac{1}{n}\Big)  \right) \\
& \less \left(\frac{1}{\ga} \log  \frac{1}{\ga}\right) \norm{\obs}_\moc \,  \moc \left( \frac{\log^4 N}{N} \right) \, .
\end{align*}
\end{remark}


\subsection*{Sharp lower bound on the rate of convergence} 
For Lebesgue a.e. frequency $\om$, we construct an example of an observable
$\obs$ showing that the Denjoy-Koksma inequality is optimal; in particular this also shows that for Diophantine frequencies, the rate of convergence $\{\rho_N\}$ of the corresponding Birkhoff averages obtained in Theorem~\ref{extension D-K} is nearly optimal (save for some logarithmic factors). Furthermore, considering frequencies with increasingly stronger Liouvillean behavior, the rate of convergence $\{\rho_N\}$ is shown to deteriorate accordingly. 

Our construction is similar in spirit to that of the Weierstrass function, a classical example of a H\"older continuous but nowhere differentiable function. 

We begin with a technical result.

\begin{proposition}\label{sharp prop} 
Let $\omega\notin \Q$ and let $\{q_k\}_{k\geq 1}$ be the  denominators of its convergents. Given $\alpha\in (0, 1]$, consider the lacunary Fourier series
$$f (x) := \sum_{k=1}^\infty \frac{1}{q_k^\alpha} \, e (q_k x) \, $$
and define the observable $\obs\colon \T \to \R$ by $ \obs (x) := \Re f (x)$.
Then the following hold.
\begin{enumerate}[(i)]
\item $\obs$ is $\alpha$-H\"older continuous.
\item Let $C < \infty$ and $c>0$ be absolute constants, with $C$ sufficiently large and $c$ sufficiently small. If for a given integer $m$ we have 
$\displaystyle q_{m+1} \ge C m \, q_m$,
then the corresponding $q_m$-th Birkhoff sum of $\obs$ over $T_\om$ satisfies
\begin{equation*}
\frac{1}{q_m} \obs^{(q_m)} ( l q_m \om) - \int_\T \obs  \ge \const \, \frac{1}{q_m^\alpha} 
\end{equation*}
for all integers $0 \le l \le c \,  \frac{q_{m+1}}{q_m} $ .
\item Furthermore, in the setting of item (ii), there is an integer $N_m \asymp q_{m+1}$ such that
\begin{equation*}
\frac{1}{N_m} \obs^{(N_m)} ( 0 ) - \int_\T \obs  \ge \const \, \frac{1}{q_m^\alpha} \, .
\end{equation*}
\end{enumerate}
\end{proposition}

\begin{proof}
Since the sequence $\{q_k\}_{k\geq 1}$ is rapidly increasing (namely $q_k \ge 2^{(k-1)/2}$), the series $\sum_{k=1}^\infty \frac{1}{q_k^\alpha} \, e (q_k x)$ converges absolutely, hence uniformly, and its sum, the function $f(x)$ is well defined and continuous. Let us prove that it is in fact $\alpha$-H\"older continuous, which will then imply the same about its real part, the observable $\obs$.

Let $x, h \in [0, 1)$. Then
\begin{align*}
\abs{f(x+h) - f (x)} & = \abs{ \sum_{k=1}^\infty \frac{1}{q_k^\alpha} \,  \left( e (q_k (x+h)) - e (q_k x) \right) } \\
& = \abs{ \sum_{k=1}^\infty \frac{1}{q_k^\alpha} \,  \left( e (q_k h) - 1 \right) \, e (q_k x)  } \le  \sum_{k=1}^\infty \frac{1}{q_k^\alpha} \,  \abs{e (q_k h) - 1} \, .
\end{align*}

We split the last sum above into two parts, depending on whether $q_k h < 1$ or $q_k h \ge 1$. 
Let $l := \min \{ k \colon q_k h \ge 1\}$. Then $\frac{1}{q_l} \le h$,  $ q_{l-1} < \frac{1}{h}$ and we have: 
\begin{align*}
\abs{f(x+h) - f (x)} & \le  \sum_{k=1}^{l-1} \frac{1}{q_k^\alpha} \,  \abs{e (q_k h) - 1} + \sum_{k=l}^{\infty} \frac{1}{q_k^\alpha} \,  \abs{e (q_k h) - 1}\\
& \less \sum_{k=1}^{l-1} \frac{1}{q_k^\alpha} \,  q_k \, h +\sum_{k=l}^{\infty} \frac{1}{q_k^\alpha} 
\less  \left( \sum_{k=1}^{l-1} q_k^{1-\alpha} \right) \, h +  \frac{1}{q_l^\alpha} \\
& \less q_{l-1}^{1-\alpha} \, h +  \frac{1}{q_l^\alpha}  \le \frac{1}{h^{1-\alpha}} \, h + h^\alpha \less h^\alpha \, ,
\end{align*}
which proves the $\alpha$-H\"older continuity of $f$, and thus establishes item (i). 

\smallskip

Note that 
$ \int_\T f = \fc{f} (0) = 0$, so $ \int_\T \obs = 0$ as well.

Using~\eqref{Birkhoff sum kernel} we have that for all $x\in \T$,
\begin{align*}
\frac{1}{q_m} f^{(q_m)} (x) - \int_\T f  & =   \sum_{k=1}^\infty \frac{1}{q_k^\alpha} \,  \esum_{q_m} (q_k \om) \, e (q_k x) \\
 &   = \frac{1}{q_m^\alpha} \,  \esum_{q_m} (q_m \om) \, e (q_m x) \\
 & + 
 \sum_{k=m+1}^\infty \frac{1}{q_k^\alpha} \,  \esum_{q_m} (q_k \om)  \, e (q_k x) 
 +  \sum_{k=1}^{m-1} \frac{1}{q_k^\alpha} \,  \esum_{q_m} (q_k \om) \, e (q_k x) \, .
\end{align*}

Taking the real part on both sides we get:
\begin{align}\label{sharp eq2}
\frac{1}{q_m} \obs^{(q_m)} (x) - \int_\T \obs =  \Sigma_m (x) + \Sigma_{> m} (x) + \Sigma_{< m} (x) \, ,
\end{align}
where
\begin{align*}
\Sigma_m (x) &:=  \frac{1}{q_m^\alpha} \, \Re \,  \esum_{q_m} (q_m \om) \, e (q_m x) \\
\Sigma_{> m} (x) &:= \sum_{k=m+1}^\infty \frac{1}{q_k^\alpha} \,  \Re \, \esum_{q_m} (q_k \om)  \, e (q_k x) \\
\Sigma_{< m} (x) &:=  \sum_{k=1}^{m-1} \frac{1}{q_k^\alpha} \, \Re \, \esum_{q_m} (q_k \om) \, e (q_k x) \, .
 \end{align*}

We will show that for $x = l q_m \om$ with $0 \le l \less \frac{q_{m+1}}{q_m}$, the term $\Sigma_m (x)$ is positive and it dominates $\Sigma_{>m}$ and $\Sigma_{< m}$.

\smallskip

Let us first bound $\Sigma_m (l q_m \om)$ from below. To do that, note that if $z\in\C$ satisfies say $\abs{z-1}<\frac{1}{2}$, then $\Re z > \frac{1}{2}$. We then set up to verify that
$$\abs{ \esum_{q_m} (q_m  \om)  \, e (q_m \, l q_m \om) - 1} \less l \, \frac{q_m}{q_{m+1}} \, .$$

Indeed, recalling~\eqref{norm nt} and~\eqref{q_n-ineq1}, for all $0\le j < q_m$ we have 
\begin{align*}
\abs{ e (j \, q_m \, \om)  \, e (q_m \, l q_m \om) - 1  } & = \abs{  e \left(  (j+q_m l)   \, q_m \om  \right) - 1 } 
 \asymp \norm{ (j +  q_m l) \, q_m \om} \\
 & \less (j + q_m l) \,  \norm{q_m \om} <  \frac{j + q_m l}{q_{m+1}} \le 2 l \,  \frac{q_m}{q_{m+1}} \, .
\end{align*}

Averaging in $j$ we have:
\begin{align*}
\abs{ \esum_{q_m} (q_m  \om)  \, e (q_m \, l q_m \om) - 1}  
= \abs{ \frac{1}{q_m} \, \sum_{j=0}^{q_m-1}  e (j \, q_m \, \om)  \, e (q_m \, l q_m \om) - 1   }  \less l \, \frac{q_{m}}{q_{m+1}} <  \frac{1}{2}  \, ,
\end{align*}
where the last equality holds provided that $l \le c \,   \frac{q_{m+1}}{q_{m}} $, for an absolute constant $c$ (which could be made explicit).

Consequently, for all $0 \le l \le c \,  \frac{q_{m+1}}{q_{m}}$  we  have
\begin{align}\label{sharp lower bound}
\Sigma_m (  l q_m \om) =  \frac{1}{q_m^\alpha} \, \Re \,  \esum_{q_m} (q_m \om) \, e (q_m \, l q_m \om) \ge \frac{1}{2} \,  \frac{1}{q_m^\alpha} \, .
\end{align}

\smallskip

Next we bound $\Sigma_{>m} (x)$ uniformly in $x\in\T$. We have
\begin{align}\label{upper bound sigma > m}
\abs{\Sigma_{>m} (x)} & \le \sum_{k=m+1}^\infty \frac{1}{q_k^\alpha} \,  \abs{ \esum_{q_m} (q_k \om)}  \, \abs{ e (q_k x) } 
 \le \sum_{k=m+1}^\infty \frac{1}{q_k^\alpha}  \notag \\
 & = \sum_{\substack{k=m+1\\k \text{ even }}}^\infty \frac{1}{q_k^\alpha}  + \sum_{\substack{k=m+1\\k \text{ odd }}}^\infty \frac{1}{q_k^\alpha}  
\less  \frac{1}{q_{m+1}^\alpha} \, ,
\end{align}
where the last inequality holds because of the geometric progression of any sequence of the form $\left\{ q_{n + 2 k}^\alpha \right\}_{k \ge 0}$.

\smallskip

Finally, using~\eqref{exp sum 1},~\eqref{norm nt} and~\eqref{q_n-ineq1}, we can estimate the exponential sums in $\Sigma_{< m}$ in a sharper way as follows:
\begin{align*}
 \abs{ \esum_{q_m} (q_k \om) } & =  \frac{1}{q_m} \, \frac{ \abs{  1- e (q_m q_k \om) }  }{\abs{  1- e (q_k \om)  }   }
 \asymp  \frac{1}{q_m} \, \frac{ \norm{q_m q_k \om} }{ \norm{q_k \om} } 
 \less  \frac{1}{q_m} \, \frac{ q_k \norm{q_m \om}}{\norm{q_k \om} } 
 \less  \frac{1}{q_m} \, \frac{q_k \, q_{k+1}}{q_{m+1}} \, .
\end{align*}  

Then
\begin{align}\label{sharp upper bound}
\abs{ \Sigma_{< m} (x) } & \le  \sum_{k=1}^{m-1} \frac{1}{q_k^\alpha} \, \abs{ \esum_{q_m} (q_k \om) }   \less
  \sum_{k=1}^{m-1} \frac{1}{q_k^\alpha} \,   \frac{1}{q_m} \, \frac{q_k \, q_{k+1}}{q_{m+1}}
  =  \sum_{k=1}^{m-1} \frac{q_k ^{1-\alpha} \, q_{k+1}}{q_m \, q_{m+1}} \notag \\
 & \le m \, \frac{q_{m-1} ^{1-\alpha} \, q_{m}}{q_m \, q_{m+1}} = m \, \frac{q_{m-1} ^{1-\alpha}}{q_{m+1}} \, .
\end{align}

Combining~~\eqref{sharp eq2},~\eqref{sharp lower bound},~\eqref{upper bound sigma > m} and~\eqref{sharp upper bound}, for all $0 \le l \le c \frac{q_{m+1}}{q_{m}}$ we obtain
$$\frac{1}{q_m} \obs^{(q_m)} (l q_m \om) - \int_\T \obs \ \more  \frac{1}{q_m^\alpha} -   \frac{1}{q_{m+1}^\alpha}  -  m \, \frac{q_{m-1}^{1-\alpha}}{q_{m+1}}  \more \frac{1}{q_m^\alpha} \, ,$$
since
$q_{m+1} \ge C m \, q_m \ge C m \, q_m^\alpha \, q_{m-1}^{1-\alpha}$ for a large absolute constant $C$.

This completes the proof of item (ii). 

\smallskip

To establish item (iii) let $N_m := (\lbar+1) q_m$, where $\lbar \asymp \frac{q_{m+1}}{q_{m}}$ is the largest integer for which the conclusion of item (ii) still holds. Thus for all $0 \le l \le  \lbar$ we have
$$ \obs^{(q_m)} (l q_m \om)  \more \frac{q_m}{q_m^\alpha} \, .$$

Then $N_m \asymp q_{m+1}$ and
\begin{align*}
 \obs^{(N_m)} (0) &= \obs^{(q_m)} (0) + \obs^{(q_m)} (q_m \om) + \ldots + \obs^{(q_m)} (\lbar q_m \om) \\
& \more (\lbar+1) \,  \frac{q_m}{q_m^\alpha}  = N_m \,  \frac{1}{q_m^\alpha} \, ,  
\end{align*}
which completes the proof of the last item.
\end{proof}

\begin{proof}[Proof of Theorem~\ref{limitations intro}]
Let $\om\in (0, 1)$ be an irrational frequency and recall that if $\om = [a_1(\om), a_2(\om), \ldots, a_m(\om), \ldots ]$ is its continued fraction expansion, and if $\{q_m (\om)\}_{m\ge1}$ are the denominators of its convergents, then for all $m$, 
$\frac{q_{m+1}(\om)}{q_m(\om)} \asymp a_m(\om)$, so
$$q_{m+1} (\om)  \more m \, q_m (\om) \Leftrightarrow 	 a_m (\om) \more m \, .$$

Since the series $\sum_{m=1}^\infty \frac{1}{m} = \infty$, by Borel-Bernstein's theorem (see~\cite[Theorem 3.2.5]{Queffelec}) we have that for almost all $\om$, the inequality
$a_m (\om) \more m$ is satisfied for infinitely many integers $m$. 

In other words, there is a full measure set in $(0, 1)$ such that for every element $\om$ of that set, there is an increasing sequence of integers $\{m_k\}_{k\ge1}$ such that
$$\frac{q_{m_k+1}(\om)}{q_{m_k}(\om)} \asymp a_{m_k}(\om) \more m_k ,$$
that is, $q_{m_k+1}(\om) \more m_k \, q_{m_k}(\om)$. 

Consider  the observable $\obs$ defined as
$$\obs(x) := \Re \, \sum_{j=1}^\infty \frac{1}{q_j^\alpha} \, e (q_j x) .$$

By Proposition~\ref{sharp prop} item (i) $\obs$ is $\alpha$-H\"older continuous and by item (ii), for all $k\ge1$,
\begin{equation*}
\abs{\frac{1}{q_{m_k}} \obs^{(q_{m_k})} (0) - \int_\T \obs } \ge \const \frac{1}{q_{m_k}^\alpha}  \, ,
\end{equation*}
which concludes the proof of the theorem.
\end{proof}

\begin{remark}\label{limitation moc} 
The same argument applies to any modulus of continuity $\moc$. Indeed, defining 
$$f (x) := \sum_{k=1}^\infty \moc \left( 1/q_k \right)  \, e (q_k x) \quad \text{and} \quad \obs (x)  := \Re f (x) \, ,$$ 
it turns out that $\obs \in C^\moc (\T)$. Moreover, if $q_{m+1} \more m q_m$, then we get
$$\abs{\frac{1}{q_m} \obs^{(q_m)} (0) - \int_\T \obs } \ge \const \, \moc \left( \frac{1}{q_m} \right) \, ,$$ 
along with the other conclusions of Proposition~\ref{sharp prop} and of Theorem~\ref{limitations intro}.
\end{remark} 

\begin{remark}
Under  more specific arithmetic assumptions on $\om$, Proposition~\ref{sharp prop} item (iii) provides stronger limitations on the rate of convergence $\{\rho_N\}$, as explained below. Let  $N_m \asymp q_{m+1}$ as in item (iii) of Proposition~\ref{sharp prop}.

$\blob$ Assume that $\om \in \sDC_{\ga}$ but $q_{m+1} \asymp q_m \log^2 (q_m+1)$ for infinitely many indices $m$ (the set of such frequencies is of course nonempty, but by Borel-Bernstein's theorem it has zero measure). 
For all of these indices we then have 
$$\frac{1}{N_m} \obs^{(N_m)} (0) - \int_\T \obs \,  \more \, \frac{1}{q_m^\alpha}
\, \asymp \, \frac{\log^{2 \alpha} q_m}{q_{m+1}^\alpha} \, \more \, \frac{\log^{2 \alpha} N_m}{N_m^\alpha} \, ,$$
thus showing that the rate of convergence~\eqref{extension D-K eq} is quite tight.

$\blob$ Similarly, assuming that $\om \in \DC_{\ga, A}$, for all indices $m$ such that    $ q_{m+1} \asymp q_m^A$, we have
$$\frac{1}{N_m} \obs^{(N_m)} (0) - \int_\T \obs  \, \more \,  \frac{1}{q_m^\alpha}
\, \asymp \, \frac{1}{q_{m+1}^{\alpha/A}} \, \asymp \, \frac{1}{N_m^{\alpha/A}} \, ,$$
thus showing that the rate of convergence~\eqref{bound DC eq} is quite tight.

$\blob$ Moreover, assuming that $\beta (\om) > 0$, there is $\beta_0>0$ and there are infinitely many indices $m$ such that $q_{m+1} \more e^{\beta_0 q_m}$. For all of these indices we have
$$\frac{1}{N_m} \obs^{(N_m)} (0) - \int_\T \obs  \, \more \, \frac{1}{q_m^\alpha}
\, \more \, \frac{1}{\log^\alpha q_{m+1}} \, \more \,  \frac{1}{\log^\alpha N_m} \, ,$$
thus showing that the rate of convergence~\eqref{bound beta eq} is quite tight.
\end{remark}

\begin{remark}
Consider any strictly decreasing function $\rho \colon (0, \infty) \to (0, \infty)$ with $\lim_{t\to\infty} \rho (t) = 0$ (arbitrarily slowly). 

Let $\Gamma (t)$ be the inverse of the function $t \mapsto \log \left( \frac{1}{\rho (t)} \right)$. Note that the slower $\rho (t) \to 0$ as $t \to \infty$, the faster $\Gamma (t) \to \infty$ as $t \to \infty$. 

Let $\om \in (0, 1)$ be a frequency such that the denominators of its convergents satisfy $q_{m+1} \asymp \Gamma (q_m)$ for all (or for infinitely many) indices $m$. Depending on how fast $\Gamma (t) \to \infty$ as $t\to\infty$, the frequency $\om$ can be extremely Liouvillean.

Define the function
$\displaystyle f (x) := \sum_{k=1}^\infty e^{- q_k} \, e (q_k x)$. Then clearly $f$ is analytic so the observable $\obs (x) := \Re f (x)$ is analytic as well.
A similar argument to the one used in the proof of Proposition~\ref{sharp prop} shows that for some $N_m \asymp q_{m+1}$ we have
$$\frac{1}{N_m} \obs^{(N_m)} (0) - \int_\T \obs  \, \more \, e^{- q_m} \, \asymp \, \rho (q_{m+1}) \, \asymp \, \rho (N_m) \, .$$

In other words, the rate of convergence $\{\rho_N\}$ of the Birkhoff averages can be arbitrarily slow for an appropriately chosen Liouvillean frequency $\om$, even with an analytic observable.
\end{remark}

%% file: translation-d.tex
We begin with a review of some Fourier analysis notions on the additive group $\T^d$, $d\ge 2$ (see~\cite[Chapter 1.9]{Katznelson} and~\cite[Chapter 3]{Grafakos_CFA} for more details).

For every multi-index $\kb\in\Z^d$, define the multiplicative characters $e_{\kb} \colon \T^d \to \C$ by $e_{\kb} (\x) := e (\kb \cdot \x)$, where 
$\kb \cdot \x = k_1 x_1 + \ldots + k_d x_d$.

Let $\obs \in  L^2 (\T^d)$ and let $\kb \in \Z^d$ be a multi-index. The corresponding Fourier coefficient of $\obs$  is then
\begin{equation*}
\fc{\obs} (\kb) = \int_{\T^d} \obs (\x) \overline{e_{\kb} (\x)} \, d \x = \int_{\T^d} \obs (\x) e (- \kb \cdot \x) \, d \x \, .
\end{equation*}

We note here that as in dimension one,  the Fourier coefficients of a function $\phi \in C^\alpha (\T^d)$ have the decay
\begin{equation}\label{decay FC d}
\abs{ \fc{\obs} (\kb) } \le \const \, \norm{\obs}_\alpha \, \frac{1}{\abs{\kb}^\alpha}  \quad \text{for all } \ \kb \in \Z^d, \, \abs{\kb} \neq 0 \, .
\end{equation}
This follows from Fubini's theorem and the corresponding one variable estimate.

Consider the Fourier series expansion of $\obs \in  L^2 (\T^d)$:
\begin{equation*}
\obs (\x) \sim  \sum_{\kb \in \Z^d} \, \fc{\obs} (\kb) \, e_{\kb}  (\x) = \int_{\T^d} \obs + \sum_{\abs{\kb} \neq 0}  \, \fc{\obs} (\kb) \, e_{\kb} (\x) \, .
\end{equation*}

Given $n\ge 0$, the $n$-th (square) partial sum of the Fourier series of $\obs$ is
$$S_n \obs (x) := \sum_{\abs{\kb} \le n} \fc{\obs} (\kb) \, e_{\kb}  (\x) \, ,$$
which, as in the one variable case, may fail to converge (even pointwise) to $\obs$. 

%

We define the $d$-dimensional (square) Jackson kernel $\Jd_n \colon \T^d \to \R$ as 
$$\Jd_n  (x_1, \ldots, x_d) := J_n (x_1) \cdot \ldots \cdot J_n (x_d) \, .$$

Then (essentially by Fubini's theorem) $\Jd_n$ has similar properties to those of its one dimensional counterpart. More precisely, 
$\abs{\fc{\Jd_n}(\kb)} \less 1$ uniformly in $n$ and $\kb$. Moreover,
$\obs_n := \obs \ast \Jd_n$ is a trigonometric polynomial in $d$ variables of degree $\le n$ and if $\obs \in C^\alpha (\T^d)$, then for all $n\ge1$,
\begin{equation}\label{Jackson type d Holder}
 \norm{\obs_n - \obs}_\infty \le \const \, \norm{\obs}_\alpha \frac{1}{n^\alpha}  \, .
\end{equation}

Furthermore, we have the following estimate on the (Fourier) coefficients of $\obs_n$: if $\kb \in \Z^d$ with $0 < \abs{\kb} \le n$ then
\begin{equation}\label{coefs obs n}
\abs{\fc{\obs_n} (\kb)}  = \abs{\fc{\obs \ast \Jd_n} (\kb)} =  \abs{\fc{\obs} (\kb)} \, \abs{\fc{\Jd_n} (\kb)} \le \const \, \norm{\obs}_\alpha \, \frac{1}{\abs{\kb}^\alpha} \, ,
\end{equation}
where the last inequality follows from~\eqref{decay FC d} and the fact that $\abs{\fc{\Jd_n} (\kb)} \less 1$.

\smallskip

We note that, as in dimension one, similar estimates to~\eqref{decay FC d}~\eqref{Jackson type d Holder} and~\eqref{coefs obs n}  hold for any modulus of continuity.

\begin{proof}[Proof of Theorem~\ref{thm d translation intro}]
Fix $N\ge 1$ and let $n = \smallo (N)$ (to be chosen later). Recall that $\obs_n = \obs \ast \Jd_n $ is  a trigonometric  polynomial of degree $\le n$ on $\T^d$, that is
$$\obs_n (\x) = \sum_{\abs{\kb} \le n} \, \fc{\obs_n} (\kb) \, e_{\kb} (\x) = \int_{\T^d} \obs_n +   \sum_{1 \le \abs{\kb} \le n} \, \fc{\obs_n} (\kb) \, e_{\kb} (\x) \, .$$

Then 
$$\frac{1}{N}  \bsumap{N}   (\x) = \int_{\T^d} \obs_n +  \sum_{1 \le \abs{\kb} \le n} \, \fc{\obs_n} (\kb) \, \frac{1}{N} \, \bsumexpk{N}  (\x) \, ,$$
so for all $\x\in \T^d$,
$$\abs{\frac{1}{N} \, \bsumap{N} \, (\x) - \int_{\T^d} \obs_n} \le \sum_{1 \le \abs{\kb} \le n} \, \abs{\fc{\obs_n} (\kb)} \, \frac{1}{N} \, \abs{ \bsumexpk{N}  (\x)} \, .$$

Let us estimate the Birkhoff sums of the multiplicative characters $e_{\kb}$. 
\begin{align*}
\bsumexpk{N}  (\x) & = \sum_{j=0}^{N-1} e_{\kb} (\x + j \omb) =  \sum_{j=0}^{N-1}  e_{\kb} (\x)  \, e_{\kb} (j \omb) \\
& = e_{\kb} (\x)  \,   \sum_{j=0}^{N-1} e (\kb \cdot j \omb) = e_{\kb} (\x)  \,   \sum_{j=0}^{N-1} e (j \, \kb \cdot \omb) \\
& = e_{\kb} (\x)  \,  \frac{1- e (N \, \kb \cdot \omb)}{1- e (\kb \cdot \omb)} \, .
\end{align*}

Hence for all $\x \in \T^d$, also using the Diophantine condition~\eqref{DC}, 
$$\abs{\bsumexpk{N} (\x) } \le \abs{ \frac{1- e (N \, \kb \cdot \omb)}{1- e (\kb \cdot \omb)}}  \le \frac{1}{\norm{\kb \cdot \omb}}  \le \frac{1}{\ga} \, \abs{\kb}^A  \, .$$

Combining this estimate on $\bsumexpk{N}$ with~\eqref{coefs obs n},  it follows that 
\begin{align*}
\norm{\frac{1}{N} \, \bsumap{N}  - \int_{\T^d} \obs_n}_{\infty} & \le \frac{\const}{\ga} \, \norm{\obs}_\alpha \, \frac{1}{N} \,  \sum_{1 \le \abs{\kb} \le n}  \frac{1}{\abs{\kb}^\alpha} \, \abs{\kb}^A 
\le   \frac{\const}{\ga} \, \norm{\obs}_\alpha \,  \frac{n^{A+d-\alpha}}{N} \, .
\end{align*}

And finally, combining the last estimate with~\eqref{Jackson type d Holder} we have:
\begin{align*}
\norm{\frac{1}{N} \, \bsum{N}  - \int_{\T^d} \obs}_{\infty} & \le \frac{\const}{\ga} \, \norm{\obs}_\alpha \,  \frac{n^{A+d-\alpha}}{N} + 
\const \, \norm{\obs}_\alpha \frac{1}{n^\alpha}  \\
& \le  \frac{\const}{\ga} \, \norm{\obs}_\alpha \, N^{ - \frac{\alpha}{A+d}} \, ,
\end{align*}
provided we choose $n := N^{\frac{1}{A+d}}$.
\end{proof}

%% file: skew-product.tex
The rate of convergence of the Birkhoff means over the affine skew product toral transformation is established using the same scheme as the one employed for the torus translation. In this case, however, the Birkhoff sums of the multiplicative characters $e_{\kb} (\x)$ on $\T^d$ are much more complex, leading to certain exponential sums of degree $d$, which are estimated using Weyl's finite differencing method (see~\cite{Montgomery}).

\begin{proof}[Proof of Theorem~\ref{skew product thm intro}]
Fix $N\ge 1$ and let $n = \smallo (N)$ (to be made explicit later). As before, let
$\obs_n := \obs \ast \Jd_n$ be the (optimal) approximation of $\obs$ by a  trigonometric polynomial of degree $n$.
From~\eqref{Jackson type d Holder} we derive that for all $\x\in\T^d$,
\begin{equation}\label{skew shift 10}
\frac{1}{N} \, \bsum{N} (x) - \int_{\T^d} \obs = \frac{1}{N} \, \bsumap{N} (x) -   \int_{\T^d} \obs_n + \bigo \left(\frac{1}{n^\alpha}\right) \, ,
\end{equation}
where the implicit constant in $\bigo \left(\frac{1}{n^\alpha}\right)$ is an absolute multiple of $\norm{\obs}_\alpha$. 
Writing 
\begin{equation*}
\obs_n (\x) = \sum_{\abs{\kb} \le n} \, \fc{\obs_n} (\kb) \, e_{\kb} (\x) = \int_{\T^d} \obs_n +   \sum_{1 \le \abs{\kb} \le n} \, \fc{\obs_n} (\kb) \, e_{\kb} (\x) \, ,
\end{equation*}
it follows that
$$\frac{1}{N} \, \bsumap{N} (x) - \int_{\T^d} \obs_n = \sum_{1 \le \abs{\kb} \le n} \, \fc{\obs_n} (\kb) \, \frac{1}{N} \, \bsumexpk{N}  (\x) \, .$$

Hence from~\eqref{coefs obs n}, for all $\x\in\T^d$,
\begin{align*}
\abs{ \frac{1}{N} \, \bsumap{N} (x) - \int_{\T^d} \obs_n  } & \le \sum_{1 \le \abs{\kb} \le n} \, \abs{\fc{\obs_n} (\kb)} \, \frac{1}{N} \, \abs{ \bsumexpk{N}  (\x)} \\
& \less \norm{\obs}_\alpha  \, \sum_{1 \le \abs{\kb} \le n} \, \frac{1}{\abs{\kb}^\alpha} \, \frac{1}{N}  \abs{ \bsumexpk{N}  (\x)}  \, .
\end{align*}

We now estimate the Birkhoff sums $\bsumexpk{N}  (\x)$  of the Fourier modes $e_{\kb} (\x)$, over the transformation $S_\om$ . We have:
$$ \bsumexpk{N}  (\x) = \sum_{j=0}^{N-1} e_{\kb} (S_\om^j \x) =  \sum_{j=0}^{N-1} e ( \kb \cdot S_\om^j \x) \, .$$

The expression above turns out to be a polynomial exponential sum, which we estimate using Weyl's theorem. 

To see how $\bsumexpk{N}  (\x)$ can be regarded as an exponential sum, let us first consider the (less technical) case $d=2$.

Given $j \in \N$, the $j$-th iterate of the transformation $S_\om$ is 
\begin{align*}
S_\om^j (x_1, x_2) & = \left(x_1+j x_2 + \frac{j (j-1)}{2} \om, x_2 + j \om\right) \\
& =  \left( \frac{\om}{2} \, j^2 + (x_2 - \frac{\om}{2}) \, j + x_1, j \, \om + x_2 \right) \, .  
\end{align*} 

Fix $\x = (x_1, x_2) \in \T^2$ and $\kb = (k_1, k_2) \in \Z^2$ with $\abs{k}\neq 0$. Then for all  $j \in \N$,
\begin{align*}
p_{\kb} (j) := \kb \cdot S_\om^j \x = \frac{k_1 \om}{2} \, j^2 + \left( k_1 (x_2 - \frac{\om}{2}) + k_2 \om  \right) \, j + k_1 x_1 + k_2 x_2 \, . 
\end{align*}

\begin{enumerate}
\item[$\blob$] If $k_1\neq0$, then $p_{\kb} (j)$ is a polynomial in $j$ of degree $2$,  with leading coefficient $a_{\kb} =  \frac{k_1}{2} \om$.  
\item[$\blob$] If $k_1 =0$ then $k_2\neq0$, so $p_{\kb} (j)$ is a polynomial in $j$ of degree $1$, with leading coefficient $a_{\kb} =  k_2 \om$.  
\end{enumerate}

Let us now consider the general case, $d\ge2$. Given $\x= (x_1, x_2, \ldots, x_d) \in \T^d$ and $j\in\N$, it is not hard to see that the $j$-th iterate of the transformation $S_\om$ is given by 
$$S_\om^j = \left(y_1^j, \, y_2^j, \, \ldots, \, y_d^j \right) \, ,$$
where
\begin{align*}
 y_{1}^{j} & = \binom{j}{0} x_1 + \binom{j}{1} x_2 + \ldots + \binom{j}{d-1} x_d + \binom{j}{d} \om \\
 y_{2}^{j} & = \binom{j}{0} x_2 + \binom{j}{1} x_3 + \ldots  +  \binom{j}{d-2} x_d + \binom{j}{d-1} \om   \\
 \vdots & \\
 y_{d}^{j} & = \binom{j}{0} x_d + \binom{j}{1} \om = x_d + j \om \, . 
\end{align*}

These formulas can be verified by induction in $j$, using the identity $\binom{j}{l} + \binom{j}{l-1} = \binom{j+1}{l}$ and the conventions $\binom{j}{0} = 1$ and $\binom{j}{l} = 0$ if $l > j$. Recall also that
$$\binom{j}{d} = \frac{j !}{(j-d) ! \, d !} = \frac{ j \, (j-1) \, \ldots \, (j-d+1) }{d !} = \frac{1}{d !} j^d + \text{ lower order terms in } j .$$ 

Fix $\x = (x_1, x_2, \ldots, x_d) \in \T^d$ and $\kb = (k_1, k_2, \ldots, k_d) \in \Z^d$ with $\abs{\kb} \neq 0$ and denote $p_{\kb} (j) := \kb \cdot S_\om^j \x$.
\begin{enumerate}
\item[$\blob$] If $k_1\neq0$, then 
$$p_{\kb} (j) = \kb \cdot S_\om^j \x = \frac{k_1}{d !} \om \, j^d + \text{ lower order terms in } j .$$ 
That is, $p_{\kb} (j) $ is a polynomial of degree $d$,  with leading coefficient $a_{\kb} =  \frac{k_1}{d !} \om$.  

\item[$\blob$] If $k_1 = 0$  but $k_2 \neq 0$, then  
$$p_{\kb} (j) = \kb \cdot S_\om^j \x = \frac{k_2}{(d-1) !} \om \, j^{d-1} + \text{ lower order terms in } j .$$ 
That is, $p_{\kb} (j) $ is a polynomial of degree $d-1$,  with leading coefficient $a_{\kb} =  \frac{k_2}{(d-1) !} \om$.  

\medskip

\noindent \ldots

\medskip

\item[$\blob$] If $\kb' := (k_1, \ldots, k_{d-1}) = (0, \ldots, 0)$, then $k_d\neq 0$ and
$$p_{\kb} (j) = \kb \cdot S_\om^j \x = k_d \om \, j + k_d x_d ,$$ 
which is a polynomial of degree $1$  with leading coefficient $a_{\kb} =  k_d \om$.  
\end{enumerate}

In this last case, using the Diophantine condition~\eqref{sDC}, we have
\begin{align*}
\abs{\sum_{j=0}^{N-1} e (p_{\kb} (j) ) } & = \abs{\sum_{j=0}^{N-1} e ( k_d \om \, j + k_d x_d ) } = \abs{\sum_{j=0}^{N-1} e ( j \, k_d \om ) }
= \abs{ \frac{ 1- e (N \, k_d \om)  }{ 1- e (k_d \om) }  } \\
& \le \frac{1}{\norm{k_d \om}} \le \frac{1}{\ga} \, \abs{k_d} \log^2 (\abs{k_d}+1) \le  \frac{1}{\ga} \, \abs{\kb} \log^2 (\abs{\kb}+1)  \le \frac{1}{\ga} \, n^{1+\smallo(1)} \, .
\end{align*}

Therefore, uniformly in $x\in\T^d$,
\begin{equation}\label{skew shift eq7}
\sum_{1 \le \abs{\kb} \le n, \, \abs{\kb'} = 0} \, \frac{1}{\abs{\kb}^\alpha} \, \frac{1}{N}  \abs{ \bsumexpk{N}  (\x)} \less  \frac{1}{\ga} \, \frac{n^{2+\smallo(1)}}{N} \, .
\end{equation}

In all other cases, let $1\le i \le d-1$ be the first index such that $k_i \neq 0$. Then $p_{\kb} (j)$ is a polynomial of degree $d_{\kb} = d-i+1 \in \{2, \ldots,  d\}$, with leading coefficient $a_{\kb} = \frac{k_i}{(d-i+1) !} \om$. Estimating the corresponding exponential sum
$$\sum_{j=0}^{N-1} e (p_{\kb} (j) )$$
is a much more subtle problem in analytic number theory. The trivial bound $N$ is of course useless, as we need a bound of order $N^{1-\beta}$, with $\beta > 0$, the larger the better. 
Weyl's differencing method reduces an exponential sum over a polynomial of degree $d$ to one over a polynomial of degree $d-1$, and eventually, by induction, to an exponential sum over a linear polynomial like above (in the case $\abs{\kb'}=0$). 

More precisely, Weyl's theorem (see Theorem 2 in~\cite[Chapter 3]{Montgomery}) implies the following: if $\frac{p}{q}$ is a principal convergent of $a_{\kb}$, then  
\begin{equation}\label{skew shift Weyl}
\abs{ \sum_{j=0}^{N-1} e (p_{\kb} (j) ) } \less_{d, \ep} N^{1+\ep} \, \left( \frac{1}{q} + \frac{1}{N} + \frac{q}{N^d} \right)^\delta \, ,
\end{equation}
where $\delta = \frac{1}{2^{d-1}}$ and $\ep>0$ is arbitrarily small.\footnote{The power $\ep$ above is related to the bound on the number of divisors function: $d (m) \less_{\ep} m^{\ep}$, which holds with any $\ep>0$. In fact, one could choose $\ep = \frac{1}{\log \log m}$. }

There is a principal convergent $\frac{p}{q}$ of $a_{\kb}$ with $q\asymp N$. Indeed, since $\om\in\sDC_\ga$, it is easy to see that $a_{\kb} = \frac{k_i}{(d-i+1) !} \om \in \sDC_{\ga'}$, where
$$\ga' = \ga \frac{1}{\abs{k_i} (d-i+1) !} \ge \ga \frac{1}{\abs{\kb} d !} = \frac{\ga}{d !} \,  \frac{1}{\abs{\kb}} \, .$$ 

Then by~\eqref{arithmetic star}, there is a principal convergent $\frac{p}{q}$ of $a_{\kb}$ with
$$N \le q \le \frac{1}{\ga'} \, N \log^2 N \le \frac{d !}{\ga} \, \abs{\kb} \, N \log^2 N .$$

By Weyl's theorem (namely by~\eqref{skew shift Weyl}),
\begin{align*}
\abs{ \sum_{j=0}^{N-1} e (p_{\kb} (j) ) } & \less_{d, \ep} N^{1+\ep} \, \left( \frac{1}{N} + \frac{d !}{\ga} \,  \frac{ \abs{\kb} \, \log^2 N}{N^{d-1} }  \right)^\delta
 \less_{d, \ga, \ep} N^{1+\ep} \, \frac{ \abs{\kb}^\delta}{N^\delta} \log^{2 \delta} N \\
& \less_{d, \ga} N^{1-\delta+\smallo(1)} \abs{\kb}^\delta \, ,
\end{align*}
where we grouped in the exponent $\smallo(1)$ any extra logarithmic type factors.

Therefore, uniformly in $x\in\T^d$,
\begin{align}\label{skew shift eq8}
\sum_{1 \le \abs{\kb} \le n, \, \abs{\kb'} \neq 0} \, \frac{1}{\abs{\kb}^\alpha} \, \frac{1}{N}  \abs{ \bsumexpk{N}  (\x)} & \less_{d, \ga} \frac{N^{1-\delta+\smallo(1)}}{N} \, \sum_{1 \le \abs{\kb} \le n, \, \abs{\kb'} \neq 0} \, \frac{1}{\abs{\kb}^\alpha} \,  \abs{\kb}^\delta \notag \\
& \le  \frac{n^{\delta-\alpha+d}}{N^{\delta-\smallo(1)}} \, .
\end{align}

Putting together~\eqref{skew shift eq7}, ~\eqref{skew shift eq8}, ~\eqref{skew shift 10}, we conclude that
\begin{align*}
\norm{\frac{1}{N} \bsum{N} - \int_{\T^d} \obs }_\infty \less_{d, \ga} \, \norm{\obs}_\alpha \, \left(  \frac{n^{2+\smallo(1)}}{N} +  \frac{n^{\delta-\alpha+d}}{N^{\delta-\smallo(1)}}  + \frac{1}{n^\alpha} \right) \, .
\end{align*}

The conclusion of the theorem follows by choosing $n = N^{\frac{\delta}{\delta+d}-\smallo(1)}$. 
\end{proof}